\documentclass[reqno,oneside,12pt]{amsart}

\usepackage{hyperref}
\usepackage{longtable}
\usepackage[ansinew]{inputenc}
\usepackage{graphicx}
\usepackage{color}
\usepackage[numbers, square]{natbib}
\usepackage{mathrsfs}
\usepackage{bbm}
\usepackage{tikz}
\usepackage{mathptmx}
\usepackage{amsmath,amsthm,amsfonts,amssymb}
\usepackage{dsfont,extarrows,enumerate}
\usepackage{fullpage}
\usepackage{xcolor}

\numberwithin{equation}{section}

\theoremstyle{plain} \newtheorem{theorem}{Theorem}[section]
\theoremstyle{plain} \newtheorem{proposition}[theorem]{Proposition}
\theoremstyle{plain} \newtheorem{lemma}[theorem]{Lemma}
\theoremstyle{plain} \newtheorem{corollary}[theorem]{Corollary}
\theoremstyle{definition} 
\theoremstyle{definition} 
\theoremstyle{remark} \newtheorem{remark}[theorem]{Remark}
\theoremstyle{remark} \newtheorem{example}[theorem]{Example}

\newcommand{\1}{\mathbbm{1}}

\newcommand{\mr}{\mathbb R}
\newcommand{\mn}{\mathbb N}

\newcommand{\eee}{{\rm e}}
\newcommand{\mmp}{\mathbb{P}}
\newcommand{\me}{\mathbb{E}}

\newcommand{\dodn}{\overset{{\rm d}}\longrightarrow}

\renewcommand{\epsilon}{\varepsilon}

\usepackage{fancybox}
\newlength{\querylen}
\begin{document}
\title{Convolution powers of unbounded measures on the positive half-line}

\author{Dariusz Buraczewski}
\address{Dariusz Buraczewski, Mathematical Institute University of Wroclaw, Pl. Grunwaldzki 2/4
50-384 Wroclaw, Poland}\email{dbura@math.uni.wroc.pl}

\author{Alexander Iksanov}
\address{Alexander Iksanov, Faculty of Computer Science and Cybernetics, Taras Shev\-chen\-ko National University of Kyiv, 01601 Kyiv, Ukraine}
\email{iksan@univ.kiev.ua}

\author{Alexander Marynych}
\address{Alexander Marynych, Faculty of Computer Science and Cybernetics, Taras Shev\-chen\-ko National University of Kyiv, 01601 Kyiv, Ukraine}
\email{marynych@unicyb.kiev.ua}

\date{\today}
\thanks{DB was supported by the National Science Center, Poland (Opus, grant number 2020/39/B/ST1/00209). AM was supported by the University of Wroclaw in the framework of the programme `Initiative of Excellence~--~Research University' (IDUB)}

\begin{abstract}
For a right-continuous nondecreasing and unbounded function $V$ of at most exponential growth, which vanishes on the negative halfline, we investigate the asymptotic behavior of the Lebesgue-Stieltjes convolution powers $V^{\ast(j)}(t)$ as both $j$ and $t$ tend to infinity. We obtain a comprehensive asymptotic formula for $V^{\ast(j)}(t)$, which is valid across different regimes of simultaneous growth of $j$ and $t$. Our main technical tool is an exponential change of measure, which is a standard technique in the large deviations theory. Various applications of our result are given. 
\end{abstract}


\maketitle

\vspace{-0.5cm}

\section{Introduction}

Let $V:\mr\to[0,\infty)$ be a right-continuous and nondecreasing function vanishing on $(-\infty,0)$. The function $V$ can be thought of as the distribution function of a measure $\mu$ on the nonnegative half-line, that is,
$$
\mu([0,\,t])=V(t),\quad t\geq 0.
$$
Let $\mu^{\ast(j)}$ (respectively, $V^{\ast(j)}$) denote the $j$-th convolution power in the Lebesgue--Stieltjes sense of measure $\mu$ (respectively, of function $V$), so that
$$
\mu^{\ast(j)}([0,\,t])=V^{\ast(j)}(t)=\int_{[0,\,t]}V^{\ast(j-1)}(t-y){\rm d}V(y),\quad t\geq 0,\quad j\geq 2.
$$
Let $(t_j)_{j\in\mn}$ be a sequence of positive numbers which diverges to $+\infty$ as $j\to\infty$.
In this note we are interested in the asymptotic behavior of $V^{\ast(j)}(t_j)$ as $j\to\infty$. If
\begin{equation}\label{eq:finite_limit_intro}
\lim_{x\to\infty}V(x)\in(0,\infty),
\end{equation}
then, by simple rescaling, we may assume that the limit is equal to $1$ and $\mu$ is a probability measure. A basic result of probability theory, called the weak law of large numbers, then tells us that, for every $\varepsilon>0$,
\begin{equation}\label{eq:classic_wlln}
\lim_{j\to\infty}V^{\ast(j)}(({\tt m}-\varepsilon)j)=0,\quad \lim_{j\to\infty}V^{\ast(j)}(({\tt m}+\varepsilon)j)=1,
\end{equation}
where ${\tt m}:=\int_{[0,\,\infty)}x{\rm d}V(x)$ is assumed finite. Under the additional assumption ${\tt s}^2:=\int_{[0,\,\infty)}x^2{\rm d}V(x)-{\tt m}^2\in (0,\infty)$, the celebrated central limit theorem provides a refinement of~\eqref{eq:classic_wlln}:
\begin{equation}\label{eq:classic_clt}
\lim_{j\to\infty }V^{\ast(j)}({\tt m} j+x{\tt s}\sqrt{j})=\frac{1}{\sqrt{2\pi}}\int_{-\infty}^{x}\eee^{-y^2/2}{\rm d}y,\quad x\in\mathbb{R}.
\end{equation}
If ${\tt s}^2=\infty$ or even ${\tt m}=\infty$, so-called `stable' versions of~\eqref{eq:classic_clt} are still available under conditions of regular variation. This involves replacing $t_j=t_j(x)={\tt m} j+x{\tt s}\sqrt{j}$ on the left-hand side with another appropriate sequence defined by the rate at which $V(t)$ tends to $1$. Further results on the asymptotic behavior of $V^{\ast(j)}(t_j)$, for other classes of sequences $(t_j)_{j\in\mathbb{N}}$, follow from classic large- and moderate-deviations theory making the picture of asymptotic regimes of $V^{\ast(j)}(t_j)$ under assumption~\eqref{eq:finite_limit_intro} is more or less complete. For this circle of questions we refer to the classic book~\cite{Petrov}.

A related problem has been addressed in the articles~\cite{Diaconis+Saloff-Coste:2014, Greville:1974, Randles+Saloff-Coste:2015, Schoenberg:1953}. These are concerned with the asymptotic behavior of $W^{\ast(j)}(t_j)$ as $j\to\infty$. Here, $W:\mathbb{Z}\to\mathbb{C}$ is a function of a finite support, and $W^{\ast(k)}$ is the $k$-fold Lebesgue convolution of $W$ with itself defined by $W^{\ast(1)}:=W$ and $W^{\ast(k)}(t):=\sum_{s\in\mathbb{Z}}W^{\ast(k-1)}(t-s)W(s)$ for $k\geq 2$ and $t\in\mathbb{C}$. In this setting the connections with probability theory vanish. One particularly interesting aspect of these works is the appearance in the asymptotic behavior of $W^{\ast(j)}(t_j)$ of the functions $H_m$ with $\int_\mr \eee^{{\rm i}zx}H_m(x){\rm d}x=\exp(-z^m)$ for $z\in\mr$ and even $m\geq 2$ and $=\exp({\rm i}z^m)$ for $z\in\mr$ and odd $m\geq 3$. If the function $W$ is nonnegative, as in our setting, then the asymptotic shape is driven by the Gaussian density $H_2$, similarly to~\eqref{eq:classic_clt}, whereas if the function $W$ is complex-valued or real-valued and takes values of both signs, then $H_m$ may arise with any $m\geq 2$.  In~\cite{Bui+Randles:2022, Randles:2023, Randles+Saloff-Coste:2017} the asymptotic behavior of $W^{\ast(j)}(t_j)$ as $j\to\infty$ was investigated for some particular subclasses of functions $W:\mathbb{Z}^d\to \mathbb{C}$ satisfying $\sum_{x\in \mathbb{Z}^d}|W(x)|<\infty$. The arguments of all the papers mentioned in this paragraph are heavily based on Fourier analysis.

Returning to our setting of real-valued nondecreasing functions $V$, assume now that the limit in~\eqref{eq:finite_limit_intro} is equal to $+\infty$, meaning that the measure $\mu$ is infinite. As far as we know, under this assumption, the problem of our interest has not been investigated so far, and the present paper aims at (partly) closing this gap. Although our results differ significantly from those obtained in the aforementioned articles, they share some similarities with the classic theory of large deviations for probability measures, see, for instance, Theorem 1 in~\cite{Bahadur+Ranga_Rao} or Theorem 1 in~\cite{Petrov_LD}.

\section{The setting}
Throughout the paper we assume that
\begin{equation}\label{eq:infinite_limit_intro}
\lim_{x\to\infty}V(x)=+\infty,
\end{equation}
and that the sequence $(t_j)_{j\in\mn}$ grows at least linearly. Also we shall normally write $t$ for $t_j$. The standing assumption in what follows is
\begin{equation}\label{eq:finite_LS_transform}
\hat{V}(s):=\int_{[0,\infty)}\eee^{-s x}{\rm d}V(x)<\infty\quad\text{for some }s>0,
\end{equation}
which means that $\hat{V}(s')$ is also finite for all $s'\geq s$. Depending on $V$ the function $\hat{V}$ may be finite for all $s>0$ or $s\geq s_1$ or $s>s_1$ for some $s_1>0$. Denote by $\mathcal{D}$ the domain of $\hat V$, that is,
$$
\mathcal{D}:=\{s>0: \hat{V}(s)<\infty\}.
$$
It is a standard fact that if $\hat{V}(s)$ is defined for $s>s_0$, then it possesses an analytic continuation to the domain $\{s\in\mathbb{C}: {\rm Re}\,s>s_0\}$, see~\cite{Widder:1941}. Thus, the 
function
$$
s\mapsto\lambda(s):=\log \hat{V}(s)=\log \int_{[0,\,\infty)}\eee^{-sx}{\rm d}V(x)
$$
is well-defined in a sufficiently small complex domain containing ${\rm Int}\,\mathcal{D}$, the interior of $\mathcal{D}$. Here $\log$ is the principal branch of the complex logarithm.

The function $\lambda$ is infinitely differentiable on ${\rm Int}\,\mathcal{D}$. In particular, $-\lambda^\prime$ is a continuous function on ${\rm Int}\,\mathcal{D}$. Assumption
\eqref{eq:infinite_limit_intro} excludes the situation in which $V(x)=a_1\1_{[a_2,\infty)}(x)$ for $x\geq 0$ and some positive $a_1$ and $a_2$. The function $\lambda$ is then strictly convex on ${\rm Int}\,\mathcal{D}$ which entails that the function $-\lambda^\prime$ is strictly decreasing on ${\rm Int}\,\mathcal{D}$. Thus, for every $s\in (s_{-},s_{+})$,  where
$$
s_{-}:=\inf_{s\in {\rm Int}\,\mathcal{D}}(-\lambda'(s))\quad\text{and}\quad s_{+}:=\sup_{s\in {\rm Int}\,\mathcal{D}}(-\lambda'(s)),
$$
the equation $-\lambda'(x)=s$ has a unique solution. According to Lemma~\ref{lem:s_minus} in the Appendix
$$
s_{-}=\inf\{x>0:\,V(x)>0\}.
$$
In most applications $s_{+}=+\infty$, and the equation $-\lambda'(x)=s$ then has a unique solution for all $s>s_{-}$. Nevertheless, the situation where $s_{+}<\infty$ is also possible, see Example~\ref{example7} below.

Throughout the paper we assume that\footnote{Note that $V^{\ast (j)}(t)=0$ for $t<s_- j$.}
\begin{equation}\label{eq:range}
t/j \in (s_{-},s_{+})\quad\text{for all large enough }j\in\mn,
\end{equation}
and let $\kappa(j)=\kappa(j,t)$ be uniquely defined by the equation
\begin{equation}\label{eq:kappa_def}
-\lambda^\prime(\kappa(j))=t/j.
\end{equation}

Here are some examples of a function $V$ and the corresponding quantities $\lambda$ and $\kappa$.  Later on some of these will be explored further.

\begin{example}\label{example1}
Let $V(x)=b x^\alpha$ for $x\geq 0$ and some positive $b$ and $\alpha$. Then $\lambda(s)=\log (b \Gamma(\alpha+1))-\alpha \log s$, $-\lambda^\prime(s)=\alpha/s$ for $s>0$, $s_{-}=0$ and $s_{+}=\infty$. Hence, $\kappa(j)=\alpha j/t$ for $j\in\mn$.
\end{example}

\begin{example}\label{example2}
Let $V(x)=ax+b$ for $x\geq 0$ and some positive $a$ and $b$. Then $\lambda(s)=\log(b+a/s)$ and $-\lambda^\prime(s)=a/(bs^2+as)$ for $s>0$. Hence, $\kappa(j,t)=((a^2+4abj/t)^{1/2}-a)/(2b)$ for $j\in\mn$.
\end{example}

\begin{example}\label{example3}
Let $V(x)=(\log (x+1))^\alpha$ for $x\geq 0$ and some $\alpha>0$. Then $\mathcal{D}={\rm Int}\mathcal{D}=(0,\infty)$. By the Karamata Tauberian theorem, see Theorem 1.7.1 in \cite{BGT}, $\hat{V}(s)~\sim~\log^{\alpha}(1/s)$ and $\lambda(s)\sim \alpha\log\log (1/s)$, as $s\to 0+$. Moreover, by the monotone density theorem (a version of Theorem 1.7.2 in \cite{BGT}),
$$
-\lambda^\prime(s)\sim \frac{\alpha}{s\log (1/s)},\quad s\to 0+.
$$
Hence, $\kappa (j)\sim \alpha/((t/j)\log (t/j))$ as $j\to\infty$ and $t/j\to \infty$.
\end{example}

\begin{example}\label{example4}
Let ${\rm d}V(x)=2^{-1}x^{-1/2}\exp(x^{1/2})\1_{(0,\infty)}(x){\rm d}x$. Then, as $s\to 0+$, $\lambda(s)\sim 1/(4s)$ and $-\lambda^\prime(s)\sim 1/(4s^2)$. Hence, $\kappa(j,t)\sim (j/t)^{1/2}/2$ as $j\to\infty$ and $t/j\to \infty$.
\end{example}

\begin{example}\label{example5}
Let $V(x)=\eee^{ax}-1$ for $x\geq 0$ and some $a>0$. Then $\lambda(s)=\log a-\log(s-a)$ and $-\lambda^\prime(s)=1/(s-a)$ for $s>a$. Hence, $\kappa(j,t)=a+j/t$ for $j\in\mn$. Interestingly, there is an explicit formula for the convolutions of $V$
\begin{equation}\label{eq:v_j_shifted_expponents_explicit}
V^{\ast(j)}(x)=\frac{a^j}{(j-1)!}\int_0^{x} y^{j-1} \eee^{ay}{\rm d}y=\frac{1}{(j-1)!}\int_0^{a x} y^{j-1} \eee^{y}{\rm d}y,\quad t\geq 0,\quad j\in\mathbb{N}
\end{equation}
which follows by induction. Here is an even simpler formula $$(V_0\ast V^{\ast(j)})(x)=\frac{(ax)^j\eee^{ax}}{j!},\quad x\geq 0,\quad j\in\mathbb{N}_0:=\mathbb{N}\cup\{0\},$$ where $V_0(x)=\eee^{ax}$ for $x\geq 0$.
\end{example}

\begin{example}\label{example6}
Let $V(x)=\eee^{ax}$ for $x\geq 0$ and some $a>0$. Then $\lambda(s)=\log s-\log (s-a)$ and $-\lambda^\prime(s)=a/(s(s-a))$ for $s>a$. Hence, $\kappa(j,t)=(a+(a^2+4aj/t)^{1/2})/2$ for $j\in\mn$.
\end{example}

\begin{example}\label{example7}
Let ${\rm d}V(x)=(x+1)^{-2-\alpha}\eee^{x}\1_{(0,\infty)}(x){\rm d}x$ for some $\alpha>0$. Then $\mathcal{D}=[1,\infty)$, $\hat{V}(1)=1/(1+\alpha)$ and $-\hat{V}'(1)=\int_0^{\infty}x(x+1)^{-2-\alpha}{\rm d}x=1/(\alpha(1+\alpha))$. Thus, $-\lambda'(1)=1/\alpha$ meaning that $s_{+}=1/\alpha$ and $\kappa(j)$ is undefined for $\alpha t>j$.
\end{example}

\section{Main results}

We call $V$ arithmetic, if it is piecewise-constant with jumps at points belonging to the lattice $(\omega_1+\omega_2 n)_{n\in\mn}$ for some $\omega_1\geq 0$ and $\omega_2>0$, and nonarithmetic, otherwise. Put $a(j):=(j\lambda^{\prime\prime}(\kappa(j)))^{1/2}$ for $j\in\mn$.
\begin{theorem}\label{thm:main} Suppose that $j\to \infty$ and $t=t_j\to\infty$ in such a way that~\eqref{eq:range} holds true.

\noindent
(a) Assume that $\lim_{j\to\infty}\kappa(j)a(j)=\infty$ and 
$\lim_{j\to\infty}\eee^{-\gamma j}\kappa(j)a(j)=0$ for all $\gamma>0$. Put
$$
T_j:=\frac{|\lambda^\prime(\kappa(j))|^3}{\lambda^{\prime\prime}(\kappa(j))}+\frac{|\hat V^{\prime\prime\prime}(\kappa(j))|}{\lambda^{\prime\prime}(\kappa(j)) \hat V(\kappa(j))}.
$$
Suppose further that
\begin{equation}\label{eq:for BerryEsseen}
T_j=o(a(j)),\quad j\to\infty
\end{equation}
and, for every fixed $\gamma>0$,
\begin{equation}\label{eq:strongnonlattice}
\limsup_{j\to\infty}\sup_{|z|\geq \gamma}\Big|\frac{\hat V(\kappa (j)-{\rm i}z/T_j)}{\hat V(\kappa(j))}\Big|<1.
\end{equation}
Then
\begin{equation}\label{eq:thm_part_a_claim}
V^{\ast(j)}(t)~\sim~\frac{1}{(2\pi)^{1/2}}\frac{(\hat V(\kappa(j)))^j\eee^{t\kappa(j)}}{\kappa(j)a(j)},\quad j\to\infty.
\end{equation}

\noindent
\noindent (b) Assume that $\lim_{j\to\infty}\kappa(j)=\theta\in (0,\infty)$, $\lim_{j\to\infty}\lambda^{\prime\prime}(\kappa(j))=\sigma^2\in (0,\infty)$ and that $V$ is nonarithmetic. Then
\begin{equation}\label{eq:thm_part_b_claim}
V^{\ast(j)}(t)~\sim~\frac{1}{(2\pi\sigma^2)^{1/2}\theta}\frac{(\hat V(\kappa(j)))^j\eee^{t\kappa(j)}}{j^{1/2}},\quad j\to\infty.
\end{equation}
\end{theorem}
\begin{remark}
Note the under the assumptions of part (b) relations~\eqref{eq:for BerryEsseen} hold automatically.
\end{remark}

Denote by $p({\tt m}, {\tt s}^2)$ the density of a normal random variable with mean ${\tt m}\in\mr$ and variance ${\tt s}^2>0$ given by
$$
p({\tt m},{\tt s}^2; %\sigma^2;
y)=\frac{1}{(2\pi{\tt s}^2)^{1/2}} %\sigma\sqrt{2\pi}}
\eee^{-\frac{(y-{\tt m})^2}{2{\tt s}^2 %\sigma^2
}},\quad y\in\mr.
$$

\begin{corollary}\label{cor:lin_growth}
Assume that $V$ is nonarithmetic and that $t=\alpha j+y$ for some fixed $\alpha>s_{-}$ and some 
$y=y(j)$. Let $\theta$ be a unique solution to $-\lambda'(\theta)=\alpha$ and $\sigma^2:=\lambda''(\theta)$.

\noindent  If $y(j)=o(j^{2/3})$ as $j\to\infty$, then
\begin{equation}\label{eq:cor_j_linear_growth}
V^{\ast(j)}(t)~\sim~\frac{1}{\theta}\eee^{\theta t}(\hat V(\theta))^j\frac{1}{(2\pi\sigma^2j)^{1/2}}\eee^{-\frac{y^2}{2\sigma^2j}}=\frac{1}{\theta}\eee^{\theta t}(\hat V(\theta))^jp(0,\sigma^2 j;y),\quad j\to\infty.
%\frac{1}{\theta}\frac{e^{\theta t}}{\sigma\sqrt{2\pi j}}\hat{V}^j(\theta)\exp\left\{-\frac{y^2}{2\sigma^2 j}\right\}=\frac{1}{\theta}e^{\theta t}\hat{V}^j(\theta)p(0,\sigma^2 j;y),\quad j\to\infty.
\end{equation}

\noindent If $y(j)\sim c j^{2/3}$ as $j\to\infty$ for some $c>0$, then
\begin{equation}\label{eq:cor_j_linear_growth2}
V^{\ast(j)}(t)~\sim~ \theta^{-1}\eee^{-c^3\lambda^{\prime\prime\prime}(\theta)/(6\sigma^6)}\eee^{\theta t}(\hat V(\theta))^jp(0,\sigma^2 j;y),\quad j\to\infty.
\end{equation}
\end{corollary}

If $\lim_{j\to\infty} j^{-2/3}y(j)=\infty$, formulae \eqref{eq:cor_j_linear_growth} and \eqref{eq:cor_j_linear_growth2} are no longer valid. Actually, in the setting of Corollary \ref{cor:lin_growth} the asymptotic behavior of $V^{\ast(j)}(t)$ exhibits infinitely many phase transitions, which are indexed by $r\in\mn$ and depend on whether $y(j)=o(j^{(r+1)/(r+2)})$ or $y(j)\sim {\rm const}\cdot j^{(r+1)/(r+2)}$ as $j\to\infty$. An inspection of the proof of Corollary \ref{cor:lin_growth} makes this observation crystal clear.

An interesting problem is to find a sequence $(t_j)_{j\in\mn}$ such that $V^{\ast(j)}(t_j)$ converges to a finite positive constant. This is accomplished in the following corollary which can also be regarded as a version of~\eqref{eq:classic_clt}.

\begin{corollary}\label{cor:clt}
Assume that there exists $\theta^{\ast}\in \mathcal{D}$ such that $\lambda(\theta^{\ast})=\theta^{\ast}\lambda'(\theta^{\ast})$ and $V$ is nonarithmetic. Fix $y\in\mr$ and let $t=t(j,y)$ be a function satisfying
\begin{equation}\label{eq:threshold}
-\lambda'\left(\theta^{\ast}-\frac{\log j - y + \varepsilon(j)}{2j\theta^{\ast}\lambda''(\theta^{\ast})}\right)=\frac{t}{j},\quad j\to\infty
\end{equation}
for some $\varepsilon(j)=o(1)$, $j\to\infty$. Then, with $\sigma^2:=\lambda''(\theta^{\ast})$,
\begin{equation}\label{cor:clt_claim}
\lim_{j\to\infty}V^{\ast(j)}(t)~=~\frac{1}{(2\pi\sigma^2)^{1/2}\theta^{\ast} %\sigma\sqrt{2\pi}
}\eee^{-y/2}.
\end{equation}
\end{corollary}

\section{Applications}

Recall our standing assumption that $V(x)=0$ for $x<0$.

\noindent
{\it Example~\ref{example1} (continuation).} Recall that in this example $V(x)=b x^{\alpha}$ for $x\geq 0$ and some $b,\alpha>0$, $\hat{V}(s)=b\Gamma(\alpha+1)s^{-\alpha}$, $-\lambda'(s)=\frac{\alpha}{s}$ and $\kappa(j)=\alpha j/t$. Further, $\lambda''(s)=\frac{\alpha}{s^2}$, which implies that
$$
\lambda''(\kappa(j))=\frac{t^2}{\alpha j^2},\quad a(j)=\frac{t}{\sqrt{\alpha j}},\quad \kappa(j)a(j)=\sqrt{\alpha j}.
$$
Condition \eqref{eq:for BerryEsseen} holds in view of 
$$
T_j~=~c_1\cdot t/j
$$
for a constant $c_1>0$. Condition~\eqref{eq:strongnonlattice} reads 
$$
\sup_{|z|>\gamma}|1-{\rm i}z|^{-\alpha}=(1+\gamma^2)^{-\alpha/2}<1\quad\text{for every fixed }\gamma>0
$$
and also holds true. Note that no assumptions have been imposed on $t=t_j$. This ensures that~\eqref{eq:thm_part_a_claim} holds true for any sequence $(t_j)_{j\in\mn}$.

If $t_j\sim c\cdot j$ for some $c>0$, the same asymptotic also follows from part (b) of Theorem~\ref{thm:main}. After simple manipulations we obtain 
$$
V^{\ast(j)}(t)~\sim~\frac{1}{(2\pi \alpha j)^{1/2}}\left(\frac{b\Gamma(\alpha+1) \eee^{\alpha }t^{\alpha}}{\alpha^{\alpha} j^{\alpha}}\right)^j,\quad j\to\infty.
$$

\vspace{2mm}

\noindent
{\it Example~\ref{example2} (continuation with $a=b=1$, Laguerre polynomials).} Recall that in this case $V(t)=t+1$ for $t\geq 0$. Put $f(t)=t$ and $g(t)=1$ for $t\geq 0$ and note that $f^{\ast(j)}(t)=t^j/(j!)$ and $g^{\ast(j)}(t)=1$ for all $j\in\mn$ and $t\geq 0$. This entails
$$V^{\ast(j)}(t)=(f+g)^{\ast(j)}(t)=\sum_{i=0}^{j}\binom{j}{i}(f^{\ast(i)}\ast g^{\ast(j-i)})(t)=\sum_{i=0}^{j}\binom{j}{i}\frac{t^i}{i!}=L_j(-t)\quad t\geq 0,$$
where $L_j$ is the standard Laguerre polynomial. Clearly,
\begin{multline*}
\hat{V}(s)=s^{-1}+1,\quad \lambda(s)=\log(s+1)-\log(s),\quad \lambda'(s)=\frac{1}{s+1}-\frac{1}{s},\quad s_{-}=0,\quad s_{+}=+\infty,\\
\lambda''(s)=\frac{1}{s^2}-\frac{1}{(1+s)^2},\quad \lambda'''(s)=\frac{2}{(1+s)^3}-\frac{2}{s^3},\quad \kappa(j)=\frac{1}{2}\left(\Big(1+\frac{4j}{t}\Big)^{1/2}-1\right).
\end{multline*}

We now distinguish three cases according to the limit behavior of the ratio $j/t$.

\noindent
{\sc Case I}: $j=o(t)$ and $j\to+\infty$. We are in the setting of Theorem~\ref{thm:main}(a), since in this case
$$
\kappa(j)=\frac{j}{t}-\frac{j^2}{t^2}+O\left(\frac{j^3}{t^3}\right),\quad j\to\infty,
$$
and, moreover,
\begin{equation}\label{ex:Laguerre1}
a(j)~\sim~t j^{-1/2},\quad \kappa(j)a(j)~\sim~\sqrt{j},\quad T_j~\sim~7t/j,\quad j\to\infty.
\end{equation}
The only condition which remains to be checked is~\eqref{eq:strongnonlattice}. Upon squaring it takes the form
$$
\limsup_{j\to\infty}\,\sup_{|z|\geq \gamma}\left(\frac{1+\frac{z^2}{(\kappa(j)+1)^2 T_j^2}}{1+\frac{z^2}{\kappa^2(j)T_j^2}}\right)<1\quad\text{for every }\gamma>0.
$$
For $0<b<a$, the function $z\mapsto (1+z^2/a)/(1+z^2/b)$ is decreasing on $(0,\infty)$. Thus,
$$
\limsup_{j\to\infty}\,\sup_{|z|\geq \gamma}\left(\frac{1+\frac{z^2}{(\kappa(j)+1)^2 T_j^2}}{1+\frac{z^2}{\kappa^2(j)T_j^2}}\right)=\limsup_{j\to\infty}\left(\frac{1+\frac{\gamma^2}{(\kappa(j)+1)^2 T_j^2}}{1+\frac{\gamma^2}{\kappa^2(j)T_j^2}}\right)=\frac{1}{1+\gamma^2/49}<1.
$$
In view of $t\kappa(j)=j/(1+\kappa(j))$, formula~\eqref{eq:thm_part_a_claim} reads
$$
L_j(-t)=V^{\ast(j)}(t)~\sim~\frac{1}{(2\pi j)^{1/2}}(G(\kappa(j)))^j,\quad j\to\infty,
$$
where
$$
G(z):=\exp\left\{\frac{1}{1+z}+\log \left(1+\frac{1}{z}\right)\right\}=\left(1+\frac{1}{z}\right)\exp\left(\frac{1}{1+z} %(1+z)^{-1}
\right),\quad z>0.
$$
\noindent
{\sc Case II}: $t\sim c \cdot j$ for some $c>0$.
In this case $\lim_{j\to\infty} \kappa(j)= \phi(c):=2^{-1}((1+4c^{-1})^{1/2}-1)>0$. Since also
\begin{multline*}
\lambda^{\prime\prime} (\kappa(j))~\to~\lambda^{\prime\prime}(\phi(c))
=\frac{1}{(\phi(c))^2}-\frac{1}{(1+\phi(c))^2}=\frac{1+2\phi(c)}{(\phi(c))^2(1+\phi(c))^2}\\
=\frac{1+2\phi(c)}{c^{-2}}=c^2(1+4c^{-1})^{1/2},\quad j\to\infty
\end{multline*}
and $V$ is nonarithmetic, part (b) of Theorem~\ref{thm:main} yields
$$
L_j(-t)=V^{\ast(j)}(t)~\sim~\frac{(c^2+4c)^{-1/4}}{\phi(c)}\frac{1}{\sqrt{2\pi t}}(G(\kappa(j)))^j,\quad j\to\infty.
$$

\noindent
{\sc Case III}: $t=o(j)$. In this case $\kappa(j)\sim (j/t)^{1/2}$ as $j\to\infty$. In particular, $\kappa(j)$ diverges to infinity. Also, 
\begin{multline*}
\lambda''(\kappa(j))~\sim~2(t/j)^{3/2},\quad a(j)~\sim~(2t^{3/2}j^{-1/2})^{1/2},\\
 \kappa(j)a(j)~\sim~2^{1/2}(tj)^{1/4},\quad  T_j~\sim 3(t/j)^{1/2},\quad j\to\infty,
\end{multline*}
and~\eqref{eq:for BerryEsseen} holds true. Unfortunately,~\eqref{eq:strongnonlattice} does not hold in this regime which means that Theorem~\ref{thm:main} is not applicable. Nevertheless, a formal substitution of the above ingredients into~\eqref{eq:thm_part_a_claim} yields
\begin{equation}\label{eq:perron1}
L_j(-t)=V^{\ast(j)}(t)~\sim~\frac{1}{2(\pi^2 tj)^{1/4}}(G(\kappa(j))^j,\quad j\to\infty,
\end{equation}
which under the additional assumption $t=o(j^{1/3})$ as $j\to\infty$ simplifies to
\begin{equation}\label{eq:perron2}
L_j(-t)=V^{\ast(j)}(t)~\sim~%\frac{1}{2\sqrt{\pi}}
\frac{1}{2(\pi^2 tj)^{1/4}}\exp\big(-t/2+2(tj)^{1/2}\big), %e^{-t/2}e^{2\sqrt{tj}},
\quad t\to\infty.
\end{equation}
Formula~\eqref{eq:perron2} is indeed correct. Actually, it is a particular case of Perron's formula for the Laguerre polynomials, see formula (8.22.3) on p.~199 in~\cite{Szogo:1975}.

\vspace{2mm}
\noindent
{\it Example~\ref{example2} (general functions of linear growth).} Consider a more general situation $V(x)=a\cdot x+\varepsilon(x)$, where $a>0$ and $x\mapsto \varepsilon(x)$ is a function of finite
total variation on $[0,\infty)$. Suppose that\footnote{By definition, $\int_{[0,\infty)}y^p |{\rm d}\varepsilon(y)|=\int_{[0,\infty)}y^p {\rm d}\Xi_{[0,y]}(\varepsilon)$, where $\Xi_{[0,y]}(\varepsilon)$ is the total variation of $\varepsilon$ on $[0,y]$.}, for some $p\in \{3,4,5,\ldots\}$,
$\int_{[0,\infty)}y^p|{\rm d}\varepsilon(y)|<\infty$ and put
$$
\beta_k:=\int_{[0,\infty)}y^k {\rm d}\varepsilon(y),\quad k=0,1,\ldots,p-1.
$$
Therefore, as $s\to 0+$, by Taylor's expansion,
\begin{align}
\hat{V}(s)&=a/s+\sum_{k=0}^{p-1}\frac{\beta_k (-s)^k}{k!}+O(s^p),\quad -\hat{V}^{\prime}(s)=a/s^2+\sum_{k=0}^{p-2}\frac{\beta_{k+1} (-s)^k}{k!}+O(s^{p-1}),\notag\\
\hat{V}^{\prime\prime}(s)&=2a/s^3+\sum_{k=0}^{p-3}\frac{\beta_{k+2} (-s)^k}{k!}+O(s^{p-2}),\quad
-\hat{V}^{\prime\prime\prime}(s)=6a/s^4+\sum_{k=0}^{p-4}\frac{\beta_{k+3}(-s)^k}{k!}+O(s^{p-3})\label{ex:linear_V_derivatives}
\end{align}
and
$$
\lambda(s)=-\log s+\log a +\sum_{k=1}^{p}\delta_k s^k+O(s^{p+1}),\quad s\to 0+,
$$
where $\delta_k$'s are the values of certain multivariate polynomials at points $\beta_{\ell}/a$, $\ell=0,\ldots,p-1$, which can be calculated explicitly. For instance, $\delta_1=\beta_0/a$, $\delta_2=-\beta_1/a-\beta_0^2/(2a^2)$, $\delta_3=\beta_2/(2a)+\beta_0\beta_1/a^2+\beta_0^3/(3a^3)$. Since the derivatives of $\lambda$ are the known rational functions of the derivatives of $\hat{V}$, we conclude from~\eqref{ex:linear_V_derivatives} that
$$
\lambda'(s)=-1/s+ \sum_{k=1}^{p}k\delta_ks^{k-1}+O(s^p),\quad \lambda''(s)\sim~1/s^2,\quad \lambda'''(s)~\sim~-2/s^3,\quad s\to 0+.
$$
Denote by $\ell$ a solution to $-\lambda^\prime(\ell(s))=1/s$ for $s>0$. In particular, $\kappa(j)=\ell(j/t)$. By asymptotic inversion, we infer
\begin{equation}\label{eq:ell_linear}
\ell(s)=s+\sum_{k=2}^{p+1}\iota_{k-1}s^k+O(s^{p+2}),\quad s\to 0+
\end{equation}
for some real coefficients $\iota_1,\ldots,\iota_p$. For instance, $\iota_1=-\delta_1$, $\iota_2=\delta_1^2-2\delta_2$, $\iota_3=-\delta_1^3+6\delta_1\delta_2-3\delta_3$. Assume that $j\to+\infty$ and $t/j\to +\infty$. A specialization of~\eqref{eq:ell_linear} yields
\begin{equation}\label{eq:kappa_linear}
\kappa(j)=j/t+\sum_{k=2}^{p+1}\iota_{k-1} (j/t)^k+O((j/t)^{p+2}).
\end{equation}
By virtue of
$$
a(j)~\sim~tj^{-1/2},\quad \kappa(j)a(j)~\sim~ j^{1/2},\quad T_j~\sim~ 7 t/j,\quad j\to\infty,
$$
relation~\eqref{eq:for BerryEsseen} holds true. Next, we check that~\eqref{eq:strongnonlattice} holds without any extra assumptions. Using that $\int_{[0,\infty)}|{\rm d}\varepsilon(y)|<\infty$ and writing
$$
\hat{V}(\kappa(j)-{\rm i} z/T_j)=\frac{a}{\kappa(j)-{\rm i}z/T_j}+\int_{[0,\infty)}\eee^{-x(\kappa(j)-{\rm i}z/T_j)}{\rm d}\varepsilon(x),
$$
we infer
$$
\sup_{|z|\geq \gamma}|\hat{V}(\kappa(j)-{\rm i} z/T_j)|\leq \sup_{|z|\geq \gamma}\frac{a}{|\kappa(j)-{\rm i}z/T_j|}+ \int_{[0,\infty)}\eee^{-x\kappa(j)}|{\rm d}\varepsilon(x)|=\frac{a}{((\kappa(j))^2+z^2/T_j^2)^{1/2}}+O(1).
$$
Thus,~\eqref{eq:strongnonlattice} follows upon noticing that
\begin{multline*}
\limsup_{j\to\infty}\sup_{|z|\geq \gamma}\frac{|\hat{V}(\kappa(j)-{\rm i} z/T_j)|}{\hat{V}(\kappa(j))}=\limsup_{j\to\infty}\frac{j}{at}\sup_{|z|\geq \gamma}|\hat{V}(\kappa(j)-{\rm i} z/T_j)|\\ \leq\lim_{j\to\infty}\frac{j}{t((\kappa(j))^2+\gamma^2/T_j^2)^{1/2}}=\frac{1}{(1+\gamma^2/49)^{1/2}}<1.
\end{multline*}
Thus, part (a) of Theorem~\ref{thm:main} is applicable. Using \eqref{eq:ell_linear} in combination with $(\lambda(\ell(s))+\ell(s)/s)^\prime=-\ell(s)/s^2$ we conclude that 
$$
\lambda(\ell(s))+\ell(s)/s=-\log s+C-\sum_{k=1}^p (\iota_k/k) s^k+O(s^{p+1}),\quad s\to 0+,$$ 
for some $C\in\mathbb{R}$. Using $\lim_{s\to 0+}(\lambda(s)+\log s)=\log a$, we conclude that $C=1+\log a$, whence
$$(\hat{V}(\kappa(j)))^j\eee^{t\kappa(j)}=\eee^{j(\lambda(\ell(j/t))+\ell(j/t)/(j/t))}=\frac{t^j}{j^j}a^j\eee^j \exp\left(-\sum_{k=1}^{p}(\iota_k/k) j^{k+1}/t^k+O(j^{p+2}/t^{p+1})\right).$$
Using Stirling's approximation we obtain from~\eqref{eq:thm_part_a_claim}
\begin{equation}\label{eq:prep}
V^{\ast(j)}(t)~\sim~\frac{t^j a^j}{j!}\exp\left(-\sum_{k=1}^p (\iota_k/k) j^{k+1}/t^k+O(j^{p+2}/t^{p+1})\right),\quad j\to\infty,
\end{equation}
which is meaningful provided that $j=o(t^{(p+1)/(p+2)})$. A particular case of this formula for $p=1$ was derived by a completely different technique in~\cite{Bohun_eta_al}.

Let $\xi_1$, $\xi_2,\ldots$ be independent copies of an almost surely (a.s.) positive random variable $\xi$. The function $U$ defined by $U(t):=1+\sum_{n\geq 1}\mmp\{\xi_1+\ldots+\xi_n\leq t\}$ for $t\geq 0$ is called renewal function. The original motivation behind the present work has been our desire to understand the asymptotic of $U^{\ast(j)}(t)$ for various asymptotic regimes of divergent $j$ and $t$. Now we explain that \eqref{eq:prep} does the job under some additional assumptions imposed on the distribution of $\xi$. Assume that $\me [\xi^{p+2}]<\infty$ for some $p\in\mn$ and that some convolution power of the distribution of $\xi$ has an absolutely continuous component. Put $\mathfrak{m}:=\me[\xi]$. By Remark 3.1.7 (ii) in \cite{Frenk:1987}, the function $x\mapsto U(x)-x/\mathfrak{m}$ has a finite total variation and, furthermore, $\int_{[0,\infty)} y^p |{\rm d}(U(y)-y/\mathfrak{m})|<\infty$. Since $$\int_{[0,\infty)}\eee^{-sy}{\rm d}(U(y)- y/\mathfrak{m})=\frac{1}{1-\me [\eee^{-s\xi}]}-\frac{1}{\mathfrak{m}s},\quad s>0,$$ an application of Lebesgue's dominated convergence theorem yields 
$$
\beta_k=\int_{[0,\infty)}y^k {\rm d}(U(y)-y/\mathfrak{m})=(-1)^k\lim_{s\to 0+}\Big(\frac{1}{1-\me [\eee^{-s\xi}]}-\frac{1}{\mathfrak{m}s}\Big)^{(k)},\quad k=0,1,\ldots, p-1.
$$ 
In particular,
\begin{align*}
\beta_0&=\me [\xi^2]/(2\mathfrak{m}^2),\quad \beta_1=\me[\xi^3]/(6\mathfrak{m}^2)-(\me[\xi^2])^2/(4\mathfrak{m}^3),\\
\beta_2&=\me [\xi^4]/(12 \mathfrak{m}^2)-\me [\xi^3]\me [\xi^2]/(3\mathfrak{m}^3)+(\me [\xi^2])^3/(4\mathfrak{m}^4).
\end{align*}
Thus, assuming, for instance, that $p=3$ we obtain from~\eqref{eq:prep}
$$U^{\ast(j)}(t)~\sim~\frac{t^j}{\mathfrak{m}^j j!}\exp\big(b_1 j^2/t+ b_2 j^3/t^2+b_3 j^4/t^3\big),\quad j\to\infty$$ whenever $j=o(t^{4/5})$. Here, $b_1=\me [\xi^2]/(2\mathfrak{m})$, $b_2=-\me [\xi^3]/(6\mathfrak{m}) %2\me [\xi^3]/(3\mathfrak{m})-5(\me[\xi^2])^2/(8\mathfrak{m}^2)
$, $$b_3=\frac{1}{24}\Big(\frac{\me [\xi^4]}{\mathfrak{m}}+\frac{2\me [\xi^2]\me[\xi^3]}{\mathfrak{m}^2}-\frac{(\me [\xi^2])^3}{\mathfrak{m}^3}\Big).$$ The expression for $b_3$ makes it clear that, unlike $b_1$ and $b_2$, $b_j$ with $j\geq 3$ is not a function of $\me [\xi^{j+1}]$ alone.

\vspace{2mm}
\noindent
{\it Example~\ref{example5} (continuation).} Recall that $V(t)=\eee^{at}-1$ and there is an explicit formula~\eqref{eq:v_j_shifted_expponents_explicit} for $V^{\ast(j)}(t)$. Further
$$
\lambda''(\kappa(j))=t^2/j^2,\quad a(j)=tj^{-1/2},\quad \kappa(j)a(j)=atj^{-1/2}+j^{1/2},\quad T_j=7t/j.
$$
Note that $\kappa(j)a(j)$ always diverges as $j\to\infty$ but in order to ensure subexponential divergence we need to assume that
\begin{equation}\label{eq:example5_subexp}
t=o(\eee^{\gamma j}),\quad j\to\infty\quad\text{ for every }\gamma>0.
\end{equation}
Condition~\eqref{eq:for BerryEsseen} holds true whereas condition~\eqref{eq:strongnonlattice} follows from the formula
$$
\frac{\hat{V}(\kappa(j)-{\rm i}z/T_j)}{\hat{V}(\kappa(j))}=\frac{1}{1-{\rm i}z/7}.
$$
Thus, Theorem~\ref{thm:main}(a) yields
\begin{multline*}
V^{\ast(j)}(t)= %V^{\ast(j)}(t)=
\frac{1}{(j-1)!}\int_0^{at} y^{j-1} \eee^{y}{\rm d}y~\sim~\frac{1}{\sqrt{2\pi}}\left(\frac{a\eee t}{j}\right)^j\eee^{at}\left(\frac{at}{\sqrt{j}}+\sqrt{j}\right)^{-1}\\~\sim~\frac{(at)^{j-1}}{(j-1)!}\eee^{at}\frac{at}{at+j},\quad j\to\infty,
\end{multline*}
where the last equivalence is a consequence of Stirling's approximation. This asymptotic formula holds under the sole assumption~\eqref{eq:example5_subexp}.

Our results have a neat interpretation in the theory of branching processes. Let $\mathcal{P}$ be a random point process on $[0,\infty)$, that is, a random element defined on some underlying probability space, which takes values in the space of locally finite point measures on $[0,\infty)$. Now we define a branching random walk generated by $\mathcal{P}$. At time $0$ there is one individual, the ancestor located at the origin. The ancestor generates offspring, constituting the first generation, whose birth times are given by the point process $\mathcal{P}$. The first generation gives rise to the second generation, where the differences in birth times between individuals and their respective mothers are distributed in accordance with copies of $\mathcal{P}$, which are independent for distinct mothers. The second generation produces the third one, and so on. Each individual in these successive generations operates independently of others.

By definition, the intensity measure of $\mathcal{P}$ is a deterministic measure given by
$$
\mu(A):=\mathbb{E}[\mathcal{P}(A)]\quad\text{ for Borel }A\subseteq [0,\infty).
$$
Thus, $\mu(A)$ is the average number of individuals of the first generation with birth times belonging to $A$. Denote by $N_j(t)$ the number of individuals of the generation $j$ with birth times in $[0,t]$. Using induction in $j$, it can be checked that $\me [N_j(t)]=\mu^{\ast(j)}([0,t])$. Our results provide a robust tool for the asymptotic analysis of these averages in generations $j$, with $j$ and $t$ tending to infinity.

For the time being, put $V(x)=\mu([0,x])$ for $x\geq 0$. Assume that there exists $\theta^\ast>0$ such that $\lambda(\theta^\ast)=\theta^\ast \lambda^\prime(\theta^\ast)$. A specialization of Proposition 3.2 in \cite{Iksanov+Mallein:2022} yields for $\theta\in {\rm Int}\,\mathcal{D}$, $\theta<\theta^\ast$ $$\lim_{j\to\infty}\sup_{y\in\mr}\Big|j^{1/2}\eee^{-\theta(-\lambda^\prime(\theta)j+y)}(\hat V(\theta))^{-j}N_j(-\lambda^\prime(\theta)j+y)-\frac{W(\theta)}{\theta}\frac{1}{(2\pi \lambda^{\prime\prime}(\theta))^{1/2}}\eee^{-\frac{y^2}{2\lambda^{\prime\prime}(\theta)j}}\Big|=0\quad \text{a.s.},$$ where $W(\theta)$ is a random variable with $\me [W(\theta)]=1$ (the a.s. limit of the uniformly integrable Biggins martingale). In particular, if $y=O(j^{1/2})$ we infer $$N_j(-\lambda^\prime(\theta)j+y)~\sim~\frac{W(\theta)}{\theta}\eee^{\theta(-\lambda^\prime(\theta)j+y)}(\hat V(\theta))^j p(0,\lambda^{\prime\prime}(\theta)j; y),\quad j\to\infty\quad\text{a.s.}$$ Formula \eqref{eq:cor_j_linear_growth} demonstrates that this a.s. convergence is accompanied by the convergence in mean. Finally, we note that we are not aware of any results on the a.s. convergence of $N_j(t)$ with diverging  $j=o(t)$.

\section{Proofs}

Denote by $X_{j,\,t}$ a random variable with the distribution
\begin{equation}\label{eq:change meas}
\mmp\{X_{j,\,t}\in{\rm d}s\}=\frac{\eee^{-\kappa(j)s}V^{\ast(j)}({\rm d}s)}{(\hat V(\kappa(j)))^j},\quad s\geq 0,\quad j\in\mn
\end{equation}
and note that
\begin{equation}\label{eq:change meas_laplace}
\me [\eee^{-u X_{j,\,t}}]=\frac{(\hat{V}(\kappa(j)+u))^j}{(\hat{V}(\kappa(j)))^j}=\eee^{j(\lambda(\kappa(j)+u)-\lambda(\kappa(j)))},\quad u\geq 0,\quad j\in\mn.
\end{equation}
The original choice of $\kappa(j)$ was motivated by this change of measure and more importantly by the fact that $\me [X_{j,t}]=t$ which can be readily seen by differentiating~\eqref{eq:change meas_laplace}. Also, observe that ${\rm Var}\,[X_{j,t}]=(a(j))^2$.  Put $\bar X_{j,t}=t-X_{j,t}$ for $j\in\mn$.

Integrating over $s\in [0,t]$ in~\eqref{eq:change meas} we obtain
$$
V^{\ast (j)}(t)=(\hat V(\kappa(j)))^j\eee^{t\kappa(j)}\me [\eee^{-\kappa(j)\bar X_{j,t}}\1_{\{\bar X_{j,t}\geq 0\}}],\quad j\in\mn.
$$
Thus, Theorem~\ref{thm:main} is equivalent to the following proposition.
\begin{proposition}\label{prop:expon}
(i) Under the assumptions of part (a) of Theorem~\ref{thm:main}
$$
\lim_{j\to\infty}\kappa(j)a(j)\me [\eee^{-\kappa(j)\bar X_{j,t}}\1_{\{\bar X_{j,t}\geq 0\}}]=p(0,1;0)=(2\pi)^{-1/2}.
$$

\noindent (ii) Under the assumptions of part (b) of Theorem \ref{thm:main}
$$
\lim_{j\to\infty} \theta j^{1/2} \me [\eee^{-\kappa(j)\bar X_{j,t}}\1_{\{\bar X_{j,t}\geq 0\}}]=p(0,\sigma^2;0)=(2\pi\sigma^2)^{-1/2}.
$$

\end{proposition}
To prove the proposition, we need a central local limit theorem for a special triangular array of random variables.
\begin{proposition}\label{prop:local}
(i) Under the assumptions of part (a) of Theorem \ref{thm:main},
$$\lim_{j\to\infty}\sup_{x\in\mr}\Big|\frac{a(j)}{\delta(j)}\mmp\{\bar X_{j,\,t}\in [xa(j), xa(j)+\delta(j))\}-p(0,1;x)\Big|=0$$ for any sequence $(\delta(j))_{j\in\mn}$ of positive numbers satisfying
$$
\lim_{j\to\infty}\frac{\delta(j)}{a(j)}=0\quad\text{and}\quad\lim_{j\to\infty}\frac{\eee^{\gamma j}\delta(j)}{a(j)}=\infty
$$
for every fixed $\gamma>0$.

\noindent (ii) Under the assumptions of part (b) of Theorem \ref{thm:main}, locally uniformly in $\delta>0$,
\begin{equation}\label{eq:Stone}
\lim_{j\to\infty}\sup_{x\in\mr}\Big|\mmp\{\bar X_{j,\,t}\in [xa(j), (x+\delta)a(j))\}-\delta p(0,1;x)\Big|=0.
\end{equation}
\end{proposition}

\subsection{Proof of Proposition~\ref{prop:local}}

If the assumptions of part (i) prevail, we argue along the lines of the proof of a local limit theorem for standard random walks (Corollary 1 in~\cite{Stone:1965}).

For $j\in\mn$ and $\delta>0$, put
$$
v_{j,\,\delta}:=\delta^{-1}\mmp\{\bar X_{j,\,t}\in [x, x+\delta)\},\quad x\in\mr.
$$
The so defined function is a version of the density of a random variable $\bar X_{j,\,t}-U_\delta$, where the variables $\bar X_{j,\,t}$ and $U_\delta$ are independent, and $U_\delta$ has a uniform distribution on $(0,\delta)$. In particular,
$$
\psi_\delta(z):=\me [\eee^{-{\rm i}zU_\delta}]=\frac{1-\eee^{-{\rm i}\delta z}}{{\rm i}\delta z},\quad z\in\mr.
$$
Thus,
\begin{equation}\label{eq:12}
\int_\mr \eee^{{\rm i}zx}v_{j,\,\delta}(xa(j)){\rm d}x=\frac{1}{a(j)}\int_\mr \eee^{{\rm i}zx/a(j)}v_{j,\,\delta}(x){\rm d}x=\frac{1}{a(j)}\me \eee^{{\rm i}z\bar X_{j,\,t}/a(j)}\psi_\delta(z/a(j)).
\end{equation}

The function $\varphi$ defined by $\varphi(z):=(1-|z|)\1_{(-1,1)}(z)$ for $z\in\mr$ is the characteristic function of an absolutely continuous probability distribution with density $g$ given by $$g(x):=\frac{1-\cos x}{\pi x^2},\quad x\in\mr,$$
see Chapter XV.2 in~\cite{Feller_Vol2}. For each $b>0$, put $g_b(x):=bg(bx)$, $x\in\mr$ and $\varphi_b(z):=\varphi(z/b)$, $z\in\mr$. Then $\varphi_b$ is the characteristic function of a distribution with density $g_b$. Using this in combination with \eqref{eq:12} we infer, for $j\in\mn$ and $b>0$,
$$
\int_\mr \eee^{{\rm i}zx}\int_\mr v_{j,\,\delta}((x-y)a(j))g_b(y){\rm d}y{\rm d}x=\frac{1}{a(j)}\me [\eee^{{\rm i}z\bar X_{j,\,t}/a(j)}]\psi_\delta(z/a(j))\varphi_b(z),\quad z\in\mr.
$$
By the Fourier inversion,
$$
\int_\mr v_{j,\,\delta}((x-y)a(j))g_b(y){\rm d}y=\frac{1}{2\pi a(j)}\int_{-b}^b \eee^{-{\rm i}zx}\me [\eee^{{\rm i}z\bar X_{j,\,t}/a(j)}]\psi_\delta(z/a(j))\varphi_b(z){\rm d}z,\quad x\in\mr.
$$
Equivalently, for $x\in\mr$,
\begin{multline}\label{eq:local_clt_proof1}
\frac{a(j)}{\delta}\int_\mr \mmp\{\bar X_{j,\,t}\in [(x-y)a(j), (x-y)a(j)+\delta)\}g_b(y){\rm d}y\\
=\frac{1}{2\pi}\int_{-b}^b \eee^{-{\rm i}zx}\me [\eee^{{\rm i}z\bar X_{j,\,t}/a(j)}]\psi_\delta(z/a(j))\varphi_b(z){\rm d}z.
\end{multline}

We shall prove that, with $b=b(j)=\eee^{cj}$ for some $c>0$ to be specified below and $\delta=\delta(j)$, where $(\delta(j))_{j\in\mn}$ is any sequence of positive numbers satisfying $\delta(j)=o(a(j))$ as $j\to\infty$,
\begin{equation}\label{eq:inter}
\lim_{j\to\infty}\sup_{x\in\mr}\Big|\frac{a(j)}{\delta(j)}\int_\mr \mmp\{\bar X_{j,\,t}\in [(x-y)a(j), (x-y)a(j)+\delta(j))\}g_{b(j)}(y){\rm d}y-p(0,1;x)\Big|=0.
\end{equation}
As a preparation, for each $j\in\mn$, let $X^\prime_{j,t}$ be a nonnegative random variable with the Laplace transform
\begin{equation}\label{eq:x_prime_def}
\me [\eee^{-uX^\prime_{j,t}}]=\frac{\hat V(\kappa(j)+u)}{\hat V(\kappa(j))},\quad u\geq 0.
\end{equation}
%Note that
Formula~\eqref{eq:change meas_laplace} implies that $\bar{X}_{j,t}$ is the sum of $j$ independent copies of $t/j-X^\prime_{j,t}$. Thus, according to Lemma 1 on p.~109 in~\cite{Petrov} we have
\begin{equation}\label{eq:berry}
\big|\me [\eee^{{\rm i}z\bar X_{j,\,t}/a(j)}]-\eee^{-z^2/2}\big|\leq 16|z|^3\eee^{-z^2/3} L_j,\quad |z|\leq (4L_j)^{-1},
\end{equation}
where
$$
L_j:=\frac{1}{j^{1/2}}\me \left[\left|\frac{X^\prime_{j,\,t}-t/j}{(\lambda^{\prime\prime}(\kappa(j)))^{1/2}}\right|^3\right].
$$
Now we check that under the present assumptions
\begin{equation}\label{eq:third_moment_to_zero}
\lim_{j\to\infty} L_j=0. 
\end{equation}
Indeed,
\begin{multline*}
\me \left[\left|\frac{X^\prime_{j,\,t}-t/j}{(\lambda^{\prime\prime}(\kappa(j)))^{1/2}}\right|^3\right]\leq \frac{4((t/j)^3+\me [(X^\prime_{j,\,t})^{3}])}{(\lambda''(\kappa(j)))^{3/2}}\\
=4\left(\frac{|\lambda'(\kappa(j))|}{(\lambda''(\kappa(j)))^{1/2}}\right)^3+ \frac{4|\widehat{V}'''(\kappa(j))|}{(\lambda''(\kappa(j)))^{3/2}\widehat{V}(\kappa(j))}=
\frac{4 T_j}{\sqrt{\lambda''(\kappa(j))}},
\end{multline*}
and the right-hand side is $o(j^{1/2})$ as $j\to+\infty$ in view of~\eqref{eq:for BerryEsseen}. This proves~\eqref{eq:third_moment_to_zero}. As a consequence of~\eqref{eq:third_moment_to_zero}, 
$X_{j,\,t}$ satisfies the central limit theorem
$$
\frac{X_{j,\,t}-t}{a(j)}~\dodn~{\rm N}(0,1),\quad j\to\infty,
$$
where ${\rm N}(0,1)$ is a random variable with the standard normal law.

Using $p(0,1;x)=(2\pi)^{-1}\int_\mr \eee^{-{\rm i}zx}\eee^{-z^2/2}{\rm d}z$ for $x\in\mr$ we obtain from~\eqref{eq:local_clt_proof1}
\begin{align*}
&\hspace{-1cm}2\pi \sup_{x\in\mr}\Big|\frac{a(j)}{\delta(j)}\int_\mr \mmp\{\bar X_{j,\,t}\in [(x-y)a(j), (x-y)a(j)+\delta(j))\}g_{b(j)}{\rm d}y-p(0,1;x)\Big|\\
&=\sup_{x\in\mr}\Big|\int_{-b(j)}^{b(j)} \eee^{-{\rm i}zx}\me [\eee^{{\rm i}z\bar X_{j,\,t}/a(j)}]\psi_{\delta(j)}(z/a(j))\varphi_{b(j)}(z){\rm d}z-\int_\mr \eee^{-{\rm i}zx}\eee^{-z^2/2}{\rm d}z\Big|\\
&\leq \int_{-(4L_j)^{-1}}^{(4L_j)^{-1}}\big|\me [\eee^{{\rm i}z\bar X_{j,\,t}/a(j)}]\psi_{\delta(j)}(z/a(j))\varphi_{b(j)}(z)-\eee^{-z^2/2}\big|{\rm d}z\\
&\hspace{3cm}+\int_{(4L_j)^{-1}<|z|\leq b(j)}|\me [\eee^{{\rm i}z\bar X_{j,\,t}/a(j)}]|{\rm d}z+\int_{|z|>(4L_j)^{-1}}\eee^{-z^2/2}{\rm d}z\\
&=:A(j)+B(j)+C(j).
\end{align*}
To analyze $A(j)$, write, for $|z|\leq (4L_j)^{-1}$,
\begin{align*}
&\hspace{-1cm}\big|\me \eee^{{\rm i}z\bar X_{j,\,t}/a(j)}\psi_{\delta(j)}(z/a(j))\varphi_{b(j)}(z)-\eee^{-z^2/2}\big|\\
&\leq \big|\me [\eee^{{\rm i}z\bar X_{j,\,t}/a(j)}]-\eee^{-z^2/2}\big|+|\varphi_{b(j)}(z)-1|\eee^{-z^2/2}+\big|\psi_{\delta(j)}(z/a(j))-1\big|\eee^{-z^2/2}\\
&\leq 16|z|^3\eee^{-z^2/3}L_j+|z|\eee^{-z^2/2}/b(j)+|z|\eee^{-z^2/2}\delta(j)/(2a(j)).
\end{align*}
We have used~\eqref{eq:berry} to obtain the first term on the right-hand side. The third term is a consequence of the inequality $|1-\eee^{-{\rm i}x}-{\rm i}x|\leq x^2/2$ for $x\in\mr$, see Lemma 1 on p.~512 in~\cite{Feller_Vol2}. Thus, recalling also~\eqref{eq:third_moment_to_zero},
$$
A(j)\leq L_j\int_\mr |z|^3\eee^{-z^2/3}{\rm d}z+\big(\eee^{-cj}+\delta(j)/(2a(j))\big)\int_\mr |z|\eee^{-z^2/2}{\rm d}z~\to~ 0,\quad j\to\infty.
$$
Now we work with $B(j)$. As we have shown, $L_j\leq 4T_j/a(j)$. Thus,
\begin{multline*}
B(j)=\int_{(4L_j)^{-1} <|u|\leq b(j)}\Big|\frac{\hat V(\kappa(j)-{\rm i}u/a(j))}{\hat V(\kappa(j))} \Big|^j{\rm d}u\leq 2b(j) \sup_{|u|\geq (4L_j a(j))^{-1}}\Big|\frac{\hat V(\kappa(j)-{\rm i}u)}{\hat V(\kappa(j))}\Big|^j\\
\leq 2b(j) \sup_{|u|\geq (16 T_j)^{-1}}\Big|\frac{\hat V(\kappa(j)-{\rm i}u)}{\hat V(\kappa(j))}\Big|^j.
\end{multline*}
By virtue of ~\eqref{eq:strongnonlattice}, for large $j\in\mn$,
$$
\sup_{|u|\geq 1/16}\Big|\frac{\hat V(\kappa(j)-{\rm i}u / T_j)}{\hat V(\kappa(j))}\Big|\leq \eee^{-c_1}
$$ for some $c_1>0$. Hence, recalling that $b(j)=\eee^{cj}$ we obtain $B(j)\leq 2\eee^{(c-c_1)j}\to 0$ as $j\to\infty$ whenever $c\in (0,c_1)$. Since $\lim_{j\to\infty}L_j=0$, we immediately conclude that $\lim_{j\to\infty} C(j)=0$. Combining fragments together we arrive at~\eqref{eq:inter}.

Let $(\rho(j))_{j\in\mn}$ be a sequence of positive numbers tending to zero, to be specified below. Replacing in~\eqref{eq:inter} first $\delta(j)$ with $\delta(j)(1\pm 2\rho(j))$ and then $x$ with % by 
$x\mp \rho(j)\delta(j)/a(j)$ yields
\begin{multline*}
\lim_{j\to\infty}\sup_{x\in\mr}\Big|\frac{a(j)}{(1\pm2\rho(j))\delta(j)}\int_\mr \mmp\{\bar X_{j,\,t}\in [(x-y)a(j)\mp\rho(j)\delta(j), (x-y)a(j)+(1\pm\rho(j))\delta(j))\}\\
\times g_{b(j)}(y){\rm d}y-p(0,1;x\mp \rho(j)\delta(j)/a(j))\Big|=0.
\end{multline*}
Thus, given $\varepsilon>0$, for all $x\in\mr$ and large enough $j$,
\begin{multline*}
D_j(x):=\frac{a(j)}{\delta(j)}\int_\mr \mmp\{\bar X_{j,\,t}\in [(x-y)a(j)+\rho(j)\delta(j), (x-y)a(j)+ (1-\rho(j))\delta(j))\}g_{b(j)}(y){\rm d}y\\
\geq (1-2\rho(j))p(0,1;x+\rho(j)\delta(j)/a(j))-\varepsilon.
\end{multline*}
Observe that, for any $x\in\mr$, $$[x-y+\rho(j)\delta(j)/a(j), x-y+(1-\rho(j))\delta(j)/a(j))\subseteq [x, x+\delta(j)/a(j))$$ whenever $|y|\leq \rho(j)\delta(j)/a(j)$. This entails
\begin{multline}\label{eq:local_clt_proof2}
D_j(x)\leq \frac{a(j)}{\delta(j)}\mmp\{\bar X_{j,\,t}\in [xa(j), xa(j)+\delta(j))\}\int_{|y|\leq \rho(j)\delta(j)/a(j)}g_{b(j)}(y){\rm d}y\\+\frac{a(j)}{\delta(j)}\int_{|y|>\rho(j)\delta(j)/a(j)}g_{b(j)}(y){\rm d}y.
\end{multline}
Recall that so far the sequence $(\delta(j))_{j\in\mn}$ has been % was 
an arbitrary sequence satisfying $\delta(j)=o(a(j))$ as $j\to\infty$. We shall now also exploit the assumption that $\delta(j)/a(j)$ decays subexponentially. Let $(\rho(j))_{j\in\mn}$ be defined by
\begin{equation}\label{eq:local_clt_proof3}
\frac{a(j)}{(\rho(j))^2\delta(j)}=b(j)=\eee^{cj},\quad j\in\mn.
\end{equation}
Such a choice and subexponential decay of $\delta(j)/a(j)$ guarantees $\lim_{j\to\infty}\rho(j)=0$. 
With this choice of $\rho$ and using that $g_{b(j)}(y)=b(j)g(b(j)y)$ we conclude from~\eqref{eq:local_clt_proof2} that
$$
D_j(x)\leq\frac{a(j)}{\delta(j)}\mmp\{\bar X_{j,\,t}\in [xa(j), xa(j)+\delta(j))\}\int_{|y|\leq 1/\rho(j)}g(y){\rm d}y\\+\frac{a(j)}{\delta(j)}\int_{|y|>1/\rho(j)}g(y){\rm d}y
$$
and thereupon
\begin{multline}
\frac{a(j)}{\delta(j)}\mmp\{\bar X_{j,\,t}\in [xa(j), xa(j)+\delta(j))\}\\\geq \Big((1-2\rho(j))p(0,1;x+\rho(j)\delta(j)/a(j))-\varepsilon-\frac{a(j)}{\delta(j)}\int_{|y|>1/\rho(j)}g(y){\rm d}y\Big)\Big(\int_{|y|\leq 1/\rho(j)}g(y){\rm d}y\Big)^{-1}.\label{eq:half}
\end{multline}
Uniform continuity of $p$ in combination with $\lim_{j\to\infty}\big(\rho(j)\delta(j)/a(j)\big)=0$ entails $$\lim_{j\to\infty}\sup_{x\in\mr}|p(0,1;x+\rho(j)\delta(j)/a(j))-p(0,1;x)|=0.$$ Since $\rho$ converges to 0 and $g$ is the density of a probability distribution, $\lim_{j\to\infty}\int_{|y|\leq 1/\rho(j)}g(y){\rm d}y=1$. In view of $\int_{|y|>1/\rho(j)}g(y){\rm d}y \leq 4\pi^{-1}\int_{1/\rho(j)}^\infty y^{-2}{\rm d}y=4\pi^{-1}\rho(j)$ and the facts that $\rho$ exhibits an exponential decay, whereas $a(j)/\delta(j)$ grows subexponentially, we conclude that $$\lim_{j\to\infty}\frac{a(j)}{\delta(j)}\int_{|y|>1/\rho(j)}g(y){\rm d}y=0.$$ 
Letting now in \eqref{eq:half} $j\to\infty$ and then $\varepsilon\to 0+$ we obtain that, uniformly in $x\in\mr$, $$\liminf_{j\to\infty} \frac{a(j)}{\delta(j)}\mmp\{\bar X_{j,\,t}\in [xa(j), xa(j)+\delta(j))\}\ge p(0,1;x).$$ The converse inequality for the upper limit
follows analogously but simpler.

Relation~\eqref{eq:Stone} is essentially Stone's local limit theorem (Corollary 1 in~\cite{Stone:1965}). In view of this we omit a proof. The only thing that needs to be mentioned is that nonarithmeticity of $V$ is equivalent to the inequality $$\bigg| \frac{\hat{V}(\theta-{\rm i}u)}{\hat{V}(\theta)}\bigg|<1,\quad u\in\mr\setminus\{0\}.$$

The proof of Proposition~\ref{prop:local} is complete.

\begin{remark}\label{rem:nonarithmeticity}
Recall that a probability measure with the characteristic function $\phi$ is called strongly nonlattice, 
if
$$
\sup_{|u|>\gamma}|\phi(u)|<1
$$
for every fixed $\gamma>0$, see, for instance, equation (3) in~\cite{Stone:1965}. Our assumption~\eqref{eq:strongnonlattice} is equivalent to
\begin{equation}\label{eq:nonlattice}
\limsup_{j\to+\infty}\sup_{|u|>\gamma}|\me [\eee^{iu (X_{j,t}^{\prime}-t/j)/T_j}]|<1
\end{equation}
for every fixed $\gamma>0$, where the random variables $X_{j,t}^{\prime}$ are as defined in~\eqref{eq:x_prime_def}. Thus,~\eqref{eq:strongnonlattice} can be regarded as the assumption that the distributions of the normalized random variables $(X_{j,t}^{\prime}-t/j)/T_j$ are asymptotically strongly nonlattice as $j\to\infty$.
\end{remark}

\subsection{Proof of Proposition~\ref{prop:expon}}

Suppose that the assumptions of part (a) of Theorem~\ref{thm:main} hold. Put $\delta(j)=\theta /\kappa(j)$ for a fixed $\theta>0$ and all $j\in\mn$. The so defined $\delta(j)$ satisfies the assumption of part (i) of Proposition~\ref{prop:local}. We use a decomposition
\begin{multline*}
\kappa(j)a(j)\me [\eee^{-\kappa(j)\bar X_{j,\,t}}\1_{\{\bar X_{j,\,t}\geq 0\}}]=\kappa(j)a(j)\int_{[0,\,\theta a(j))}\eee^{-\kappa(j)y}{\rm d}\mmp\{\bar X_{j,\,t}\leq y\}\\+\kappa(j)a(j)\int_{[\theta  a(j),\,\infty)}\eee^{-\kappa(j)y}{\rm d}\mmp\{\bar X_{j,\,t}\leq y\}=:K(j)+L(j).
\end{multline*}
In view of $\lim_{j\to\infty}\kappa(j)a(j)=\infty$,
$$
L(j)\leq \kappa(j)a(j)\eee^{-\theta \kappa(j) a(j)}~\to~0,\quad j\to\infty.
$$
We proceed by writing
\begin{align*}
K(j)&\leq \kappa(j)a(j)\sum_{i=0}^{\lfloor \kappa(j)a(j)\rfloor-1}\int_{[i\theta, (i+1)\theta)}\eee^{-y}{\rm d}\mmp\{\bar X_{j,\,t}\leq y/\kappa(j)\}\\
&\leq \theta\sum_{i=0}^{\lfloor \kappa(j)a(j)\rfloor-1}\eee^{-i\theta}\Big(\frac{a(j)}{\delta(j)}\mmp\{\bar X_{j,t}\in [i\delta(j), (i+1)\delta(j))\}-p(0,1;i\delta(j)/a(j))\Big)\\
&\hspace{0.6cm}+\theta \sum_{i=0}^{\lfloor \kappa(j)a(j)\rfloor-1}\eee^{-i\theta}\big(p(0,1;i\delta(j)/a(j))-p(0,1;0)\big)+p(0,1;0)\theta\sum_{i=0}^{\lfloor a(j)\kappa(j)\rfloor-1}\eee^{-i\theta}\\
&=:K_1(j,\theta)+K_2(j,\theta)+K_3(j,\theta).
\end{align*}
Note that
$$
\lim_{j\to\infty}K_3(j,\theta)=p(0,1;0)\frac{\theta}{1-\eee^{-\theta}}\to p(0,1;0),\quad \theta\to 0+.
$$
Further, given $\varepsilon>0$ and picking $\theta>0$ sufficiently small, we obtain, for large $j\in\mn$ and $i\in [0, \lfloor \kappa(j)a(j)\rfloor-1]$,
$$
|p(0,1;i\delta(j)/a(j))-p(0,1;0)|\leq \varepsilon.
$$
This yields $\limsup_{\theta\to 0+}\limsup_{j\to\infty}|K_2(j,\theta)|\leq \varepsilon$. Therefore, $\limsup_{\theta\to 0+}\limsup_{j\to\infty}|K_2(j,\theta)|=0$, since $K_2(j,\theta)$ does not depend on $\varepsilon$. Finally, by Proposition~\ref{prop:local}, given $\varepsilon>0$, for large $j\in\mn$ and all $i\in [0,\lfloor \kappa(j)a(j)\rfloor-1]$,
$$
\Big|\frac{a(j)}{\delta(j)}\mmp\{\bar X_{j,\,t}\in [i\delta(j), (i+1)\delta(j))\}-p(0,1;i\delta(j)/a(j))\Big|\leq \varepsilon.
$$
As a consequence, $\lim_{j\to\infty}K_1(j,\theta)=0$.

We have proved that $\limsup_{j\to\infty} (K(j)+L(j))\leq p(0,1;0)=(2\pi)^{-1/2}$. The converse inequality for the lower limit follows analogously. The proof of Proposition~\ref{prop:expon} is complete under the assumptions of part (a) of Theorem \ref{thm:main}.

Under the assumptions of part (b) of Theorem~\ref{thm:main}, $\lim_{j\to\infty}\kappa(j)=\theta\in (0,\infty)$ and $\kappa(j)a(j)\to \infty$ holds automatically. The proof goes by exactly the same reasoning as in case (a), and we omit further details.

\subsection{Proofs of Corollaries~\ref{cor:lin_growth} and~\ref{cor:clt}.}

\begin{proof}[Proof of Corollary~\ref{cor:lin_growth}]
Note that $t/j\to \alpha>s_{-}$ as $j\to\infty$ which entails $\lim_{j\to\infty}\kappa(j)=\theta$.

We start by proving \eqref{eq:cor_j_linear_growth}. According to part (b) of Theorem~\ref{thm:main}, it suffices to check that under the assumption $y(j)=o(j^{2/3})$
\begin{equation*}
(\hat{V}(\kappa(j)))^j\eee^{t\kappa(j)}~\sim~(\hat{V}(\theta))^j \eee^{\theta t}\eee^{-\frac{y^2}{2\sigma^2j}}, \quad j\to\infty.
\end{equation*}
Put $\widetilde{\kappa}(j)=\theta-\kappa(j)$. Using \eqref{eq:kappa_def} and Taylor's expansion we obtain
\begin{multline*}
t/j=\alpha+y/j=-\lambda^{\prime}(\theta-\widetilde{\kappa}(j))=-\lambda'(\theta)+\sigma^2\widetilde{\kappa}(j)+O((\widetilde{\kappa}(j))^2)\\
=\alpha+\sigma^2\widetilde{\kappa}(j)+O((\widetilde{\kappa}(j))^2),\quad j\to\infty.
\end{multline*}
Thus,
$$
\widetilde{\kappa}(j)+O((\widetilde{\kappa}(j))^2)=\frac{y}{\sigma^2 j}
$$
and thereupon
\begin{equation}\label{eq:cor_proof1}
\widetilde{\kappa}(j)=\frac{y}{\sigma^2 j}+O\left(\frac{y^2}{j^2}\right),\quad j\to\infty.
\end{equation}
It remains to show that
\begin{equation}\label{eq:cor_proof2}
j\lambda(\theta-\widetilde{\kappa}(j))+t(\theta-\widetilde{\kappa}(j))
=\theta t+\lambda(\theta)j-\frac{y^2}{2\sigma^2 j}+o(1),\quad j\to\infty.
\end{equation}
This follows by another application of Taylor's expansion
\begin{align*}
&\hspace{-0.1cm} %j(\lambda(\theta-\widetilde{\kappa}(j))-(\theta-\widetilde{\kappa}(j))\lambda'(\theta-\widetilde{\kappa}(j)))\\
%&~=~
j\lambda(\theta-\widetilde{\kappa}(j))+t(\theta-\widetilde{\kappa}(j))\\
&~=~j\left(\lambda(\theta)-\lambda'(\theta)\widetilde{\kappa}(j)+\frac{\sigma^2}{2}(\widetilde{\kappa}(j))^2+O\left((\widetilde{\kappa}(j))^3\right)\right)+\theta t -t\widetilde{\kappa}(j)\\
&~\overset{\eqref{eq:cor_proof1}}{=}~\theta t + \lambda(\theta)j-y\widetilde{\kappa}(j)+\frac{\sigma^2j}{2}(\widetilde{\kappa}(j))^2
+O\left(\frac{y^3}{j^2}\right)\\
&~\overset{\eqref{eq:cor_proof1}}{=}~\theta t + j\lambda(\theta)-\frac{y^2}{\sigma^2 j}+\frac{j\sigma^2}{2}\left(\frac{y^2}{\sigma^4j^2}+O\left(\frac{y^3}{j^3}\right)\right)+O\left(\frac{y^3}{j^2}\right)~=~\theta t+j\lambda(\theta)-\frac{y^2}{2j\sigma^2}+o(1).
\end{align*}
Here, $O(y^3/j^2)=o(1)$ is secured by the assumption $y(j)=o(j^{2/3})$. This finishes the proof of both~\eqref{eq:cor_proof2} and \eqref{eq:cor_j_linear_growth}.

The proof of \eqref{eq:cor_j_linear_growth2} goes the same path and only requires an extra term in Taylor's expansions. A counterpart of \eqref{eq:cor_proof1} is
\begin{equation*}
\widetilde{\kappa}(j)=\frac{y}{\sigma^2 j}+\frac{\lambda^{\prime\prime\prime}(\theta)y^2}{2\sigma^6j^2}+o\left(\frac{y^2}{j^2}\right),\quad j\to\infty.
\end{equation*}
With this at hand,
\begin{align*}
&\hspace{-1cm} j\lambda(\theta-\widetilde{\kappa}(j))+t(\theta-\widetilde{\kappa}(j))\\
&~=~j\left(\lambda(\theta)-\lambda'(\theta)\widetilde{\kappa}(j)+\frac{\sigma^2}{2}(\widetilde{\kappa}(j))^2-\frac{\lambda^{\prime\prime\prime}(\theta)}{6}(\widetilde{\kappa}(j))^3+O\left((\widetilde{\kappa}(j))^4\right)
\right)+\theta t -t\widetilde{\kappa}(j)\\
%&~= ~\theta t + \lambda(\theta)j-y\widetilde{\kappa}(j)+\frac{j\sigma^2}{2}(\widetilde{\kappa}(j))^2-\frac{\lambda^{\prime\prime\prime}(\theta)j}{6}(\widetilde{\kappa}(j))^3+O\left(\frac{y^3}{j^2}\right)\\
&~=~\theta t + \lambda(\theta)j-\frac{y^2}{\sigma^2 j}-\frac{\lambda^{\prime\prime\prime}(\theta)y^3}{2\sigma^6j^2}+\frac{\sigma^2j}{2}\left(\frac{y^2}{\sigma^4j^2}+\frac{\lambda^{\prime\prime\prime}(\theta)y^3}{\sigma^8 j^3}\right)-\frac{\lambda^{\prime\prime\prime}(\theta)y^3}{6\sigma^6 j^2}+o(1)\\
&~=~\theta t+\lambda(\theta)j-\frac{y^2}{2\sigma^2j}-\frac{c^3\lambda^{\prime\prime\prime}(\theta)}{6\sigma^6}+o(1),\quad j\to\infty.
\end{align*}
The proof of Corollary \ref{cor:lin_growth} is complete.
\end{proof}
\begin{proof}[Proof of Corollary~\ref{cor:clt}]
First of all, note that the function $t(j,y)$ exists (at least for large enough $j$), for~\eqref{eq:threshold} implies that, for large enough $t>0$, $t/j\geq
-\lambda'(\theta^{\ast})\in (s_{-},s_{+})$.

Formula~\eqref{eq:kappa_def} implies that
$$
\kappa(j)=\theta^{\ast}-\frac{\log j - y +\varepsilon(j)}{2j\theta^{\ast}\lambda''(\theta^{\ast})}=:\theta^{\ast}-\widetilde{\kappa}(j).
$$
In particular, $\lim_{j\to\infty} \kappa(j)=\theta^\ast$. According to part (b) of Theorem~\ref{thm:main}, it suffices to check that under the assumption of Corollary~\ref{cor:clt}
\begin{equation}\label{cor2:proof1}
\lim_{j\to\infty}\frac{(\hat{V}(\kappa(j)))^j\eee^{t\kappa(j)}}{j^{1/2}}=\lim_{j\to\infty}\frac{\eee^{j(\lambda(\kappa(j))-\kappa(j)\lambda'(\kappa(j))})}{j^{1/2}}=\eee^{-y/2}.
\end{equation}
Using Taylor's expansion yields 
\begin{align*}
&\hspace{-1cm}\lambda(\kappa(j))-\kappa(j)\lambda'(\kappa(j))=\lambda(\theta^{\ast}-\widetilde{\kappa}(j))-(\theta^{\ast}-\widetilde{\kappa}(j))\lambda'(\theta^{\ast}-\widetilde{\kappa}(j))\\
&=\lambda(\theta^{\ast})-\lambda'(\theta^{\ast})\widetilde{\kappa}(j)-(\theta^{\ast}-\widetilde{\kappa}(j))(\lambda'(\theta^{\ast})-\widetilde{\kappa}(j)\lambda''(\theta^{\ast}))+O\left(\frac{(\log j)^2}{j^2}\right)\\
&=\theta^{\ast}\sigma^2\widetilde{\kappa}(j)+O\left(\frac{(\log j)^2}{j^2}\right),\quad j\to\infty.
\end{align*}
This entails 
\begin{multline*}
-2^{-1}\log j + j(\lambda(\kappa_j(t))-\kappa_j(t)\lambda'(\kappa_j(t)))\\
=-2^{-1}\log j + \theta^{\ast}\sigma^2 j\widetilde{\kappa}(j)+O\left(\frac{(\log j)^2}{j}\right)=-y/2+o(1),\quad j\to\infty
\end{multline*}
and thereupon~\eqref{cor2:proof1}.
\end{proof}

\section{Appendix}

The following result is known if $V$ is a distribution function, see, for example, Chapter 9.1.2 in~\cite{Borovkov}. We have not been able to locate in the literature its version for arbitrary nondecreasing unbounded functions. Therefore, we give a proof.

\begin{lemma}\label{lem:s_minus}
Let $V:\mr\to[0,\infty)$ be a right-continuous and nondecreasing function vanishing on $(-\infty,0)$. Put $x_0:=\inf\{x>0:\,V(x)>0\}$ and $\lambda(s)=\log \int_{[0,\infty)}\eee^{-sx}{\rm d}V(x)$ for $s\in\mathcal{D}$. Then
$$
s_{-}=\inf_{s\in\mathcal{D}}(-\lambda'(s))=\lim_{s\to+\infty}(-\lambda'(s))=x_0.
$$
\end{lemma}
\begin{proof}
The second equality follows from the monotonicity of $-\lambda'$. For the proof of the last equality we distinguish two cases: $x_0=0$ and $x_0>0$.

\noindent
{\sc Case $x_0=0$.} Fix $\varepsilon>0$ and note that $\int_{[0,\varepsilon/2]}{\rm d}V(x)=:K(\varepsilon)>0$. Since
$$
\frac{\int_{(\varepsilon,\infty)}\eee^{-sx}{\rm d}V(x)}{\int_{[0,\infty)}\eee^{-sx}{\rm d}V(x)}\leq
\frac{\eee^{-s\varepsilon/2}\int_{(0,\infty)}\eee^{-sx/2}{\rm d}V(x)}{\int_{[0,\varepsilon/2]}\eee^{-sx}{\rm d}V(x)}\leq \frac{\eee^{-s\varepsilon/2}\int_{(0,\infty)}\eee^{-sx/2}{\rm d}V(x)}{\eee^{-s \varepsilon/2}K(\varepsilon)}~\to~0,\quad s\to\infty
$$
by the monotone convergence theorem, we conclude that,
$$
\int_{[0,\infty)}\eee^{-sx}{\rm d}V(x)~\sim~\int_{[0,\varepsilon]}\eee^{-sx}{\rm d}V(x),\quad s \to \infty.
$$
Similarly, one can check that
$$
\int_{[0,\infty)}x\eee^{-sx}{\rm d}V(x)~\sim~\int_{[0,\varepsilon]}x\eee^{-sx}{\rm d}V(x),\quad s \to \infty
$$
for every fixed $\varepsilon>0$. Thus,
$$
\lim_{s\to+\infty}(-\lambda'(s))=\lim_{s\to+\infty}\frac{\int_{[0,\infty)}x\eee^{-sx}{\rm d}V(x)}{\int_{[0,\infty)}\eee^{-sx}{\rm d}V(x)}=\lim_{s\to+\infty}\frac{\int_{[0,\varepsilon]}x\eee^{-sx}{\rm d}V(x)}{\int_{[0,\varepsilon]}\eee^{-sx}{\rm d}V(x)}\leq \varepsilon.
$$
Since $\varepsilon>0$ is arbitrary we conclude that $\lim_{s\to+\infty}(-\lambda'(s))=0=x_0$.

\noindent
{\sc Case $x_0>0$.} This situation can be reduced to the previous one by passing to the function $V_1(x):=V(x+x_0)$. Indeed,
$$
\lambda(s)=\log \int_{[0,\infty)}\eee^{-sx}{\rm d}V(x)=\log \int_{[x_0,\infty)}\eee^{-sx}{\rm d}V(x)=\log \int_{[0,\infty)}\eee^{-s(x+x_0)}{\rm d}V_1(x)=:-s x_0 + \lambda_1(s)
$$
and thereupon
$$
\inf_{s\in\mathcal{D}}(-\lambda'(s))=x_0+\inf_{s\in\mathcal{D}}(-\lambda_1'(s))=x_0
$$
according to the already considered case $x_0=0$.
\end{proof}

\noindent \textbf{Acknowledgment.} This work was mainly done during numerous visits of A.~Ik\-san\-ov and A.~Ma\-ry\-nych to Wroclaw in 2019-2023. They gratefully acknowledge financial support and hospitality.

\end{document}